\documentclass[12pt, a4paper, oneside]{book}
\usepackage{amssymb}
\usepackage{amsmath}
\usepackage{amsthm}
\usepackage{theoremref}
\usepackage{titling}
\usepackage{enumerate}
\usepackage{a4wide}
\usepackage{fancyhdr}
\usepackage{yfonts}
\usepackage{mathtools}
\usepackage{comment}
\usepackage{xcolor}

\title{Properties of the Ahlfors function}

\author{Anna-Mariya Otsetova}

\date{June 11, 2021}

\begin{document}

\pagestyle{fancy}
\fancyhf{}
\fancyhead[EL]{\nouppercase\leftmark}
\fancyhead[OR]{\nouppercase\rightmark}
\fancyhead[ER,OL]{\thepage}

\newtheorem{thm}{Theorem}[chapter]
\newtheorem*{thm*}{Theorem}  
\newtheorem*{conj*}{Conjecture}  
\newtheorem{lem}[thm]{Lemma}
\newtheorem{cor}[thm]{Corollary}
\newtheorem{prop}[thm]{Proposition}
\theoremstyle{definition}
\newtheorem{ex}[thm]{Example}
\newtheorem*{remark}{Remark}
\newtheorem{mydef}[thm]{Definition}

\maketitle

\frontmatter

\chapter*{Abstract}
Extremal problems in analysis have been studied for more than a hundred years. The so called $\textit{Ahlfors function}$ is the solution to one such problem. In this thesis, we shall study several properties of this function in a most general domain $\Omega$. We will consider the Ahlfors function in the context of $\textit{analytic capacity}$, and also relate it to Riemann maps of simply connected domains onto the unit disk. Our methods involve $\textit{Gelfand representation}$ of commutative Banach algebras and standard techniques form complex analysis.

\chapter*{Introduction}
Extremal problems have a rich historical background. These problems consist in finding the extrema of functions or functionals. In 1854, B. Riemann stated the now called $\textit{Riemann mapping theorem}$. He constructed a holomorphic bijective mapping between two arbitrary simply connected domains, different from the whole plane. The function he built appears as a solution of an extremal problem, i.e., finding an analytic function $f$ whose derivative at a point $z_0$ is maximized over a subset of analytic functions. Such extremal problems reappear throughout complex analysis. Almost a hundred years after Riemann, in 1947, Lars Ahlfors was working on a problem by Painlevé, on a possible extension of Liouville's theorem but for subsets of the complex plane. A key concept involved is known as the $\textit{analytic capacity}$ of a compact set. This led to the discovery of the now called $\textit{Ahlfors function}$, a solution to a similar extremal problem: finding an analytic function of modulus bounded by 1 that maximizes the derivative at a given point.

In this thesis, we shall explore the Ahlfors function and its properties in a most general domain. We will prove first its existence and uniqueness. We will show some norm-preserving results, and the "almost" surjectivity of the Ahlfors function. We shall also provide a proof of the non-separability of $H^{\infty}(\Omega)$ based on the properties of the Ahlfors function. The Ahlfors function is, of sorts, a generalization of the Riemann map. Therefore, we will also show that for a simply connected domain, the Ahlfors function coincides with the Riemann map. As mentioned previously, it is the key component in analytic capacity, which enables us to determine if a subset of the complex plane is removable, i.e., the only bounded analytic functions of this set are constants.

For the purposes of this thesis, we employ techniques from complex and functional analysis. We will use $\textit{Gelfand representation}$ of commutative Banach algebras in the specific case of $H^{\infty}(\Omega)$. By exploiting the space's algebraic structure and the isometry provided by the so-called $\textit{Gelfand transform}$, we will be able to introduce an embedding of the $H^{\infty}$ space, and work in its $\textit{maximal ideal space}$ instead. This embedding will equip us with a crucial and necessary property: weak star compactness of the unit ball of $H^{\infty}$. The results of this thesis are based on the exposition in \cite{fisher}.

\tableofcontents

\mainmatter

\chapter{Preliminaries}
In this chapter, we introduce certain concepts which will be used extensively in the main part of the thesis.
\section{Results from Complex Analysis} 
\begin{thm} \label{maxmod}
(Maximum modulus principle) Let $h(z)$ be a complex-valued harmonic function on a bounded domain $D$ such that $h(z)$ extends continuously to the boundary $\partial D$ of $D$. If $|h(z)|\leq M$ for all $z\in\partial D$, then $|h(z)|\leq M$ for all $z\in D$.
\end{thm}

\begin{lem}\label{schwarz}
(The Schwarz Lemma) Let $f:\mathbb{D}\rightarrow\mathbb{D}$ be holomorphic with $f(0)=0$. Then,
\begin{enumerate}
    \item $|f(z)|\leq |z|$ for all $z\in\mathbb{D}$
    \item If for some $z_0\neq 0$ we have $|f(z_0)|=|z_0|$, then $f$ is a rotation.
    \item $|f'(0)|\leq 1$, and if equality holds then $f$ is a rotation.
\end{enumerate}
\end{lem}
\begin{thm}
(Montel's theorem) If a family $\mathcal{F}$ of analytic functions on $D\subset\mathbb{C}$ is uniformly bounded on each compact subset of $D$, then every sequence in $\mathcal{F}$ has a uniformly convergent subsequence on each compact subset of $D$.
\end{thm}
\begin{thm}\label{riemannmap}
(The Riemann mapping theorem) Let $D$ be a non-empty, simply connected and open subset of the complex plane $\mathbb{C}$, which is not all of $\mathbb{C}$. Then, there exists a unique bijective, holomorphic map $f$ from $D$ onto the open unit disk.
\end{thm}

\begin{thm}\label{Fatou}(Fatou's theorem) Let $f$ be a bounded analytic function in the unit disc. Then for almost every $t\in [0,2\pi)$, the limit $$f(e^{it})=\lim_{r\to 1^-}f(re^{it})$$
exists and $$\|f\|_\infty=\text{\rm ess sup}_{t\in [0,2\pi)}|f(e^{it})| .$$\end{thm}

\section{Definitions and results from Functional Analysis}
\subsection{Linear functionals and the Hahn-Banach theorem} \label{linear functionals}
\begin{mydef}
A linear functional $l$ is a mapping of a linear space X over a field \textbf{F} into \textbf{F}, that is additive and homogenous, i.e., it satisfies
\begin{equation}
    l(x+y) = l(x) + l(y), \forall x,y \in X
\end{equation}
and
\begin{equation}
    l(ax) = al(x), \forall a \in \textbf{F}
\end{equation}
\end{mydef}
\begin{ex}
Let X be the the space of all integrable functions, and let \textbf{F} be the reals. Then, $I(f) = \int_{a}^{b}f(x)dx$ is a linear functional from X to $\mathbb{R}$. We can verify the linearity of this functional as follows:
\begin{equation}
    I(f+g) = \int_{a}^{b}(f(x)+g(x))dx = \int_{a}^{b}f(x)dx + \int_{a}^{b}g(x)dx = I(f)+I(g)
\end{equation}
and
\begin{equation}
    I(af) = \int_{a}^{b}af(x)dx = a\int_{a}^{b}f(x)dx = aI(f)
\end{equation}
where $f,g \in X$ and $a \in \mathbb{R}$.
\end{ex}
The basic result about linear functionals is the Hahn-Banach theorem which we state below for vector spaces over the field of complex numbers.
 
\begin{thm} \label{Hahn-Banach}

(Complex Hahn-Banach) Let $X$ be a linear space over the complex numbers. Let $p$ be a real valued function satisfying

\begin{enumerate}
    \item $p(ax) = |a|p(x)$, $\forall a \in \mathbb{C}, x \in X$
    
    \item $p(x+y) \leq p(x) + p(y)$
\end{enumerate}

Let $Y$ be a subspace of $X$, on which a linear functional $l$ is defined, satisfying

\begin{equation}
    |l(y)| \leq p(y), \forall y \in Y
\end{equation}
Then, we can extend $l$ to the whole of $X$ such that 
\begin{equation}
    |l(x)| \leq p(x), \forall x \in X
\end{equation}
\end{thm}

The proof of this theorem is omitted here, but can be found in \cite{func}.
\subsection{Dual spaces} Given a normed vector space $(X,\|\cdot\|)$ over $\mathbb{C}$ we consider the continuous linear functionals $l:X\to\mathbb{C}$. The following result characterizes such functionals.
\begin{prop} A linear functional $l:X\to\mathbb{C}$ is continuous if and only if it is bounded, i.e., there exists $C>0$ such that \begin{equation}\label{bounded-func}|l(x)|\le C\|x\|,\end{equation} for all $x\in X$.\end{prop}
We denote by $X'$ the vector space of all  continuous linear functionals on $X$. Using the above result we can introduce a norm of $X'$ by
\begin{equation}\label{funct-norm} \|l\|=\inf\{C>0:~ \eqref{bounded-func}\text{ holds }\}=\sup_{\|x\|=1}|l(x)|.\end{equation}
It turns out that $X'$ is a Banach space with respect to this norm. It is  called the {\em dual} of $X$.
\subsection{The weak-star topology and the Banach-Alaoglu theorem}
One of the major difficulties in dealing with infinite dimensional normed spaces is the fact that bounded subsets are not necessarily relatively compact in the norm topology. This hurdle can, in some sense, be removed by using weaker topologies than the one induced by the norm. In this paragraph we will focus on dual spaces and present the famous theorem of Banach and Alaoglu.
\begin{mydef} Let $X$ be a normed space. 
The weak-star topology on its  X is the smallest topology on X'  such that  maps $l\to l(x)$, where $x\in X$ is fixed, are continuous.
\end{mydef}
The topology can be described with help of the following basis of neighborhoods
$$V_{\varepsilon,x_1,\ldots,x_n}(l)=\{m\in X':~ |m(x_j)-l(x_j)|<\varepsilon,~1\le j\le n\},$$
where $l\in X'$ and $\varepsilon>0,~x_1,\ldots,x_n \in X$ are fixed but arbitrary.
\begin{thm}(Banach-Alaoglu) The unit ball in $X'$, $\{l\in X':~\|l\|\le 1\}$, is compact in the weak-star topology.
\end{thm}
Note that every $l\in X',~\|l\|\le 1$,  is a function from $X$ to $\mathbb{C}$ such that for every $x\in X$, the value $l(x)$ belongs to the closed disc $D_x$ in $\mathbb{C}$ centered at the origin, with radius $\|x\|$. In other words $l$ can be identified with an element of the cartesian product $$P=\prod_{x\in X}D_x.$$
By Tychonov's theorem $P$ is compact, and the heart of the proof consists in showing that the image of the unit ball of $X'$ by this embedding is closed in $P$.

\subsection{Gelfand's theory of Commutative Banach Algebras} \label{gelfandsection}
The purpose of the present paragraph is to present Gelfand's representation  of commutative   Banach algebras as spaces of continuous complex-valued functions on a compact Hausdorff  topological space.

\begin{mydef}\label{algebra}
An algebra is a vector space equipped with an associative  bilinear product. The algebra is commutative if the corresponding product is.  
\end{mydef}

\begin{mydef}\label{Banach-alg}
 A Banach algebra $(\mathcal{L},\|\cdot\|)$ is a normed associative algebra over the complex numbers which is complete and whose norm satisfies
\begin{equation}
    \Vert {xy}\Vert \leq \Vert x \Vert\Vert y \Vert, \forall x,y \in \mathcal{L}. 
\end{equation}
We say that $e\in \mathcal{L}$ is a  unit if $ex=x,~x\in \mathcal{L}$ and $\|e\|=1$.
\end{mydef}

Throughout in what follows $\mathcal{L}$ will be a commutative Banach algebra with the (unique) unit $e$.

\begin{mydef}\label{spectrum} The spectrum of $x\in \mathcal{L}$ is denoted by $\sigma(x)$ and is  the set of complex numbers $\lambda$ such that $\lambda e-x$ is not invertible in $\mathcal{L}$.\end{mydef}
The {\em spectral radius} of $x\in \mathcal{L}$ is denoted by $r(x)$ and defined by $r(x)=\sup_{\lambda\in\sigma(x)}|\lambda|$. It can be computed from the formula
$$r(x)=\limsup_{n\to\infty}\|x^n\|^{\frac1{n}}=\lim_{n\to\infty}\|x^n\|^{\frac1{n}}.$$
In general, $r(x)\le\|x\|$. A necessary and sufficient condition for equality is\begin{equation}\label{sp-radius-norm}\|x^2\|=\|x\|^2, \quad x\in \mathcal{L}.\end{equation}
  The basic theorem about spectra is as follows.
\begin{thm} For every $x\in \mathcal{L}$ its spectrum $\sigma(x)$ is a non-void compact subset of $\mathbb{C}$.\end{thm}

We are particularly interested in {\em ideals} of such  algebras defined as follows.
\begin{mydef}\label{ideal}
The subset $\mathcal{I}$ of $\mathcal{L}$ is an ideal if it satisfies the following properties:
\begin{enumerate}
    \item $\mathcal{I}$ is a linear subspace of $\mathcal{L}$ \\
    
    \item For any $\textbf{x} \in \mathcal{L}$, $\textbf{x}\mathcal{I} \subset \mathcal{I}$ \\
    
    \item $\mathcal{I}$ is neither $\{0\}$, nor all of $\mathcal{L}$ \\
\end{enumerate}
A maximal ideal is one that is not contained in any other ideal.
\end{mydef}
It is a direct consequence of Zorn's lemma that maximal ideals exist, in fact, any ideal is contained in a maximal ideal.
Some properties of ideals are given below. Recall that  if  $X$ is a Banach space and $Y$ is a closed subspace then the quotient $X/Y=\{x+Y:~x\in X\}$ is a Banach space w.r.t. the quotient norm $$\|x+Y\|=\inf_{y\in Y}\|x+y\|.$$
\begin{lem}(i) An ideal cannot contain invertible elements.\\
(ii) The closure of  an ideal $\mathcal{I}$ is either an  ideal, or the whole algebra. Maximal ideals are closed.\\
(iii)  If $\mathcal{I}$ is a closed ideal  then $\mathcal{A}=\mathcal{L}/\mathcal{I}$ is a commutative Banach algebra with unit
$[e]=e+\mathcal{I}$  w.r.t. to the product $$[x][y]=[xy]\Leftrightarrow (x+\mathcal{I})(y+\mathcal{I})=xy+\mathcal{I},$$
and the quotient map $x\to [x]= x+\mathcal{I}$ is an algebra-homomorphism. 
\end{lem}
The main theorem about ideals of such algebras is the so-called Gelfand-Mazur theorem which we state in the appropriate form for our purposes.
\begin{thm} \thlabel{isomorphism} If $\mathcal{I}$ is a maximal ideal in $\mathcal{L}$, then the algebra  $\mathcal{L}/\mathcal{I}$ is isomorphic to the field of complex numbers.
\end{thm}

Maximal ideals can be identified with non-zero  multiplicative linear functionals on the algebra in question.
\begin{mydef} \thlabel{multiplicativefunc}
A multiplicative  linear  functional $p$ is an algebra homomorphism of $\mathcal{L}$ into $\mathbb{C}$. The set of non-zero multiplicative linear functionals on $\mathcal{L}$ is denoted by $\mathcal{M}(\mathcal{L})$.
\end{mydef}
A remarkable property of these homomorphisms is that they are automatically continuous. 
\begin{prop}\label{contractive hom}
If  $p\in \mathcal{M}(\mathcal{L})$ then $p(e)=1$ and $p$  is a contraction, i.e.,
\begin{equation}
    |p(\textbf{x})| \leq |\textbf{x}|, \textbf{x} \in \mathcal{L}
\end{equation}
\end{prop}
The relation to maximal ideals comes from the fact that if  $p\in \mathcal{M}(\mathcal{L})$, its null-space $$\ker p=\{x\in\mathcal{L}:~p(x)=0\}$$
has codimension $1$ in $\mathcal{L}$. Hence it is a maximal subspace and, in particular, it is a maximal ideal. Conversely, for a maximal ideal $\mathcal{I}$ we can compose the quotient map $x\to x+\mathcal{I}$ with the isomorphism provided by the Gelfand-Mazur theorem to obtan an element of  $\mathcal{M}(\mathcal{L})$. Thus we arrive at the following result. 
\begin{cor}\label{max-ideal} The map  $\mathcal{M}(\mathcal{L})\ni p\to \ker p$ is a bijection from  $\mathcal{M}(\mathcal{L})$ onto the set of maximal ideals of $\mathcal{L}$.\end{cor}
For this reason  $\mathcal{M}(\mathcal{L})$ is also called the {\em maximal ideal space of} $\mathcal{L}$. 
Note that  $\mathcal{M}(\mathcal{L})$ is a subset of the dual $\mathcal{L}'$ of the Banach space $\mathcal{L}$, and by Proposition \ref{contractive hom},  $\mathcal{M}(\mathcal{L})$ is a subset of the unit sphere in $\mathcal{L}'$. Then using the Banach-Alaoglu theorem one can prove the following basic result.
\begin{thm}\label{max-id-compact} $\mathcal{M}(\mathcal{L})$ is compact in the weak-star topology on  $\mathcal{L}'$.\end{thm}
Given a compact Hausdorff topological space $S$ we shall denote by $C(S)$ the Banach space of all continuous complex-valued functions $f:S\to\mathbb{C}$ with the usual norm 
$$\|f\|_{C(S)}=\max_{s\in S}|f(s)|.$$
As an application of the results above, we obtain a map from $\mathcal{L}$ into $C(\mathcal{M}(\mathcal{L}))$ defined by 
$$\mathcal{L}\ni x\to \hat{x}, \quad \hat{x}(p)=p(x).$$
 \begin{mydef} \thlabel{Gelfand-transf} The map $x\to \hat{x}$ is called the Gelfand transform of $\mathcal{L}$.\end{mydef}
The Gelfand transform has the following properties.
\begin{thm}\label{G-transf-prop} (i) $x\to\hat{x}$ is an algebra-homomorphism.\\
(ii) For $x\in \mathcal{L}$ we have  $\hat{x}(\mathcal{M}(\mathcal{L}))=\sigma(x)$.\\
(iii) For $x\in \mathcal{L}$ we have $$\|\hat{x}\|_{C(\mathcal{M}(\mathcal{L}))}\le \|x\|,$$
(iv) If $\mathcal{L}$ satisfies \eqref{sp-radius-norm} then the Gelfand transform is an isometry of $\mathcal{L}$ onto a closed sub-algebra of  $C(\mathcal{M}(\mathcal{L}))$, that is, 
$$\|\hat{x}\|_{C(\mathcal{M}(\mathcal{L}))}= \|x\|, \quad x\in \mathcal{L}.$$
\end{thm}
Part (iv) of the above theorem has an  application which will play a central role in the sequel,. The Hahn-Banach theorem combined with the Riesz-Markov-Kakutani theorem yield a very useful repressentation of continuous linear functionals on such Banach algebras, since they can be extended to continuous linear functional on the full space $C(\mathcal{M}(\mathcal{L}))$.
\begin{cor}\label{meas-repr} For every $l\in \mathcal{L}'$ there exists a finite Borel measure   $\mu$ on  the compact space $\mathcal{M}(\mathcal{L})$, with  total variation $\|l\|$, such that 
$$l(x)=\int\hat{x}d\mu,\quad x\in \mathcal{L}.$$
\end{cor}
We close this paragraph with two relevant examples of maximal ideal spaces. 
\begin{ex}\label{max-id-1} Let $\mathcal{L}=C([a,b])$ be the Banach algebra of continuous functions on $[a,b]\subset\mathbb{R}$ with the usual supremum norm. Then clearly, for each $c\in [a,b]$, the evaluation map $p_c(f)= f(c),~f\in C([a,b])$ is a non-zero
multiplicative linear functional. We claim that $\mathcal{M}(C([a,b]))=\{p_c:~c\in [a,b]\}$. To this end, let $p\in \mathcal{M}(C([a,b]))$, let $f_0\in C([a,b])$ be the identity function $f_0(t)=t$ and apply Theorem \ref{G-transf-prop} (ii) to conclude that $c=p(f_0)\in \sigma(f_0)=[a,b]$. But then, for any polynomial $g$, we have $p(g)=g(c)$, and since polynomials are dense in $C([a,b])$, we obtain $p(f)=f(c),~f\in C([a,b])$.
\end{ex}

\begin{ex}\label{max-id-2} Let $\mathcal{L}=C_b((a,b))$ be the Banach algebra of {\em bounded} continuous functions on $(a,b)\subset\mathbb{R}$ with the supremum norm
$$\|f\|=\sup_{t\in (a,b)}|f(t)|.$$
Again, for each $c\in (a,b)$, the evaluation map $p_c(f)= f(c),~f\in C([a,b])$ is a non-zero
multiplicative linear functional on this algebra, but $\mathcal{M}(C_b((a,b)))$ contains many other elements determined by the behaviour of these functions near the end points. To see this, let $(c_n)$ be a strictly increasing sequence in $(a,b)$ with $c_n\to b$.
If $f_0$ is the identity function as above, we can cover the set $P=\{p_{c_n}:~n\ge 1\}$ with the disjoint union of neighborhoods
$$V_n=\{l\in (C_b((a,b)))':~|l(f_0)-p_{c_n}(f_0)|<\min\{c_n-c_{n-1},c_{n+1}-c_n\}\},\quad c_0=c_1$$
which implies that $P$ cannot be compact in the weak-star topology. Then $P$ cannot be closed in $\mathcal{M}(C_b((a,b)))$, hence it must have at least an adherent point $p_0\in \mathcal{M}(C_b((a,b)))\setminus  P$. It is easy to verify that $p_0(f_0)=b$, hence $p_0$  is not an evaluation at any point in $(a,b)$.
\end{ex}

\chapter{The Ahlfors function}
This chapter is devoted to  the study of an extremal problem in the algebra of  bounded analytic functions in a general domain. The exposition is based on results from \cite{fisher}.

\begin{mydef} 
Let  $\textbf{S}^2$  denote the Riemann sphere and let $\Omega\subset \textbf{S}^2$ be an open connected set.
$H^{\infty}(\Omega)$ is the space of bounded analytic functions in the domain $\Omega$ with the norm $$\|f\|_\infty=\sup_{z\in \Omega}|f(z)|,\quad f\in H^\infty(\Omega).$$ 
\end{mydef}

It is easy to verify that $H^\infty(\Omega)$ is  a Banach algebra which satisfies 
\begin{equation}\label{square}
    \|f^2\|_\infty=\|f\|_\infty^2,\quad  f\in H^{\infty}(\Omega).
\end{equation}

In what follows wwe shall denote by $H'$ the dual space of $H^{\infty}(\Omega)$ and by $\textgoth{M}$ the maximal ideal space of $H^{\infty}(\Omega)$. Recall from section \ref{linear functionals} that maximal ideals can be identified with non-zero multiplicative linear functions, in particular we regard $\textgoth{M}$ as a subset of the unit sphere of $H'$.   By Theorem \ref{max-id-compact}, $\textgoth{M}$ is compact in the weak star topology on $H'$.

Evaluations at points on $\Omega$ are the
simplest examples of multiplicative linear functionals. The same argument as in Example
\ref{max-id-2} shows that $\textgoth{M}$ contains many other multiplicative linear functionals which appear as
generalized limits of functions in $H^{\infty}(\Omega)$ at points in $\partial\Omega$.

We denote by $C(\textgoth{M})$ the Banach space of all continuous complex-valued functions $u: \textgoth{M}\rightarrow \mathbb{C}$ with the usual norm
\begin{equation}
    \|u\|_{C(\textgoth{M})} = \max_{m\in\textgoth{M}}|u(m)|
\end{equation}
As in section \ref{linear functionals}, we  define  the {\em Gelfand Transform} $\hat{f}\in C(\textgoth{M})$  of $f\in H^{\infty}(\Omega)$  by $$\hat{f}(m)=m(f),\quad m\in \textgoth{M}.$$
 By \eqref{square} and part (iv) of   Theorem \ref{G-transf-prop} it follows that the Gelfand transform is an isometry from 
 $H^{\infty}(\Omega)$ into $C(\textgoth{M})$.            
This leads to the following final remark. Using \ref{meas-repr},  for every $h\in H'$, there exists a finite Borel measure $\mu$,on the compact, Hausdorff topological space $\textgoth{M}$, with total variation $\|h\|$, such that one can write
\begin{equation} \label{representation}
    h(x) = \int_{\textgoth{M}}\hat{x}d\mu, x\in H^{\infty}(\Omega)
\end{equation}               

\section{Existence of the Ahlfors function}
We now turn to an important question regarding the existence of non-constant  functions in $H^{\infty}(\Omega)$. In the case of constants, we say that $H^\infty(\Omega)$ is trivial. The question whether $H^\infty(\Omega)$ is trivial or not is difficult to answer for arbitrary domains in $\mathbb{C}$, or $\textbf{S}^2$. Using a conformal, map we can always assume that the complement of the domain in question is a compact subset of $\mathbb{C}$. The ''amount'' of non-constant bounded analytic functions outside a fixed compact set is measured by its {\em  analytic capacity}, which we now define.

\begin{mydef}\label{defcapdef}
Let $B$ be a compact set in the extended complex plane. Then,
\begin{equation} \label{defcap}
    \gamma(B) = \sup\{|f'(\infty)|:f\in H^{\infty}(\textbf{S}^2\setminus B), \|f\|_{\infty}\leq 1, f(\infty)=0\}
\end{equation} 
is the analytic capacity of $B$, and $f'(\infty):=\lim_{z\rightarrow\infty}z(f(z)-f(\infty))$.
\end{mydef}
We shall later prove that $H^{\infty}(\Omega)$ is either a space of constants or a non-separable space.

We are interested in solutions of the extremal problems similar to the one used in the definition of   analytic capacity.  More precisely,  given $p\in\Omega$ fixed, we consider  the  extremal problem

\begin{equation}
    \gamma = \sup\{|h'(p)|: h\in H^{\infty}(\Omega), \Vert{h}\Vert_{\infty} \leq 1\}
\end{equation}

and seek for a function $h\in H^{\infty}(\Omega)$ with
\begin{enumerate}
    \item $h \in H^{\infty}(\Omega)$ \\
    \item $\Vert{h}\Vert_{\infty}=1$ \\
    \item $h'(p) = \gamma$
\end{enumerate}
We shall call such functions $\textit{extremal}$.

\begin{lem}
At least one extremal exists.
\end{lem}
\begin{proof}
We use a normal families argument. First, we assume that $\gamma$ is non-zero. Otherwise, any constant would be an extremal. Since $\gamma$ is defined as a supremum at a given point p, there exists a sequence $(h_n)$ in $H^{\infty}(\Omega)$ with $\|h_n\|_{\infty}\leq 1$ such that $\lim_{n\rightarrow\infty}|h_n'(p)| = \gamma$. Note that $\{h_n\}$ forms a normal family because this sequence is uniformly bounded by 1 in $\Omega$. Then, by Montel's theorem, there is a subsequence $h_{n_{k}}(z) \rightarrow h(z)$, uniformly on compact subsets of $\Omega$. Uniform convergence of analytic functions on compact sets implies the uniform convergence of their derivatives on compact sets. Hence,

\begin{equation}
    |h'(p)| = \lim_{n\rightarrow \infty}|h'_{n_{k}}(p)| = \gamma
\end{equation}
Since we assumed $\gamma \neq 0$, $h$ is not constant. We must now show that $\|h\|_{\infty}=1$. By the pointwise convergence of the subsequence $h_{n_{k}}$, since $h$ is the limit function, $h$ is also bounded by 1. Now assume that 
$$\|h\|_{\infty} = \sup_{z\in\Omega}|h(z)| < 1$$
Then, we can divide by this supremum, i.e., if $\|h\|_{\infty} = 1/d$ where $d>1$, 
$$\|h\|_{\infty}/(1/d) = d\|h\|_{\infty} = \|f\|_{\infty}\leq1$$
where $f\in H^{\infty}(\Omega)$ is $dh$, a new function which has a bigger derivative than $h$ and is therefore a candidate for the extremal. This contradiction shows that $\|h\|_{\infty}=1$.

\end{proof}
Let us turn now to uniqueness. For this purpose, we need some notions.
\begin{mydef}
An extreme point of a convex set is a point which does not lie in any open (excluding the endpoints) line segment joining two points of this convex set.
\end{mydef}

\begin{lem} \label{unique}
$F$ is an extreme point of the unit ball of $H^\infty$ if the only function $g\in H^{\infty}(\Omega)$ satisfying $\|F\pm g\|_{\infty}\leq 1$ is $g=0$.
\end{lem}
\begin{proof}
Assume that $F$ is an extreme point of the unit ball of $H^{\infty}(\Omega)$. Then, if $g\in H^{\infty}(\Omega)$ with $\|F\pm g\|_{\infty}\leq 1$, we have $$F=\frac{F+g}{2}+\frac{F-g}{2},$$
so that $F$ cannot be an extreme point unless $F+g=F-g$, i.e. $g=0$. To prove the converse, assume that $F$ is not an extremal. Then, we can rewrite it as a convex combination of two other functions with norm at most 1, i.e.,

$$F = tF_1 +(1-t)F_2$$
and
$$\|F_1\|_{\infty}=1,\|F_2\|_{\infty}=1$$
where $t\in(0,1)$. Assume, w.l.o.g., that $t>1/2$. Thus, we have that $F$ is closer to $F_1$. We want to show the existence of some non-vanishing $g\in H^{\infty}(\Omega)$ satisfying $\|F\pm g\|_{\infty}\leq 1$. Thus, let $g = F_1 - F$. Then,

$$\|F + g\|_{\infty} = \|tF_1 + (1-t)F_2 + F_1 - F\|_{\infty} = \|F_1\|_{\infty}\leq 1$$
since $F_1$ is inside of the unit ball of $H^{\infty}(\Omega)$. Also,

$$\|F - g\|_{\infty} = \|tF_1 + (1-t)F_2 - F_1 + F\|_{\infty} = \|2F-F_1\|_{\infty}$$ 
$$= \|2((t-1/2)F_1)+ (1-t)F_2\|_{\infty} \leq \|2|t-1/2|+2|1-t|\|_{\infty}\leq 1$$
because $2|t-1/2|+2|1-t|=1$ for $t\in(1/2, 1)$.
\end{proof}

\begin{thm}
There is a unique extremal and it vanishes at the point $p$. This extremal is called the $\textit{Ahlfors function}$ corresponding to $\Omega$ and $p$.
\end{thm}
\begin{proof}
First, we show that $F(p) = 0$. To this end, define the auxiliary function 
\begin{equation}
    g(z) = [F(z) - F(p)][1- F(z)\overline{F(p)}]^{-1}, z\in\Omega
\end{equation}
obtained by composing a conformal automorphism of the unit disc with $F$. In particular, it follows that $\|g\|_\infty=1$. We compute the derivative of $g$, which yields
\begin{equation}
    g'(z) = \frac{F'(z)(1- F(z)\overline{F(p)})+(F(z)-F(p))(\overline{F(p)}F'(z))}{(1- F(z)\overline{F(p)})^2} 
\end{equation}
\begin{equation}
    = \frac{F'(z)}{(1- F(z)\overline{F(p)})} + \frac{(F(z)-F(p))}{(1- F(z)\overline{F(p)})}\frac{(\overline{F(p)}F'(z))}{(1- F(z)\overline{F(p)})} 
\end{equation}
We can calculate $g'(z)$ at the point $p$.
\begin{equation}
    g'(p) = \frac{F'(p)}{1 - |F(p)|^2} = \gamma(1 - |F(p)|^2)^{-1}
\end{equation}
Since $\|g\|_\infty=1$, we must have $|g'(p)|\le\gamma$, which obviously implies that $F(p)=0$.
\newline
\newline
The second part of the proof is devoted to the uniqueness assertion. To this end, it will be sufficient to prove that the extremal function $F$ is an extreme point of the unit ball of $H^\infty(\Omega)$, by Lemma \ref{unique}.  Indeed, assume this for a momment and suppose that  $F_1$ and $F_2$ are extremals, then so is $(F_1+F_2)/2$, which, by definition cannot be an extremal point of the unit ball of $H^\infty(\Omega)$. Let then $g\in H^{\infty}(\Omega)$ satisfy $\|F \pm g\|_{\infty}\leq 1$. Now squaring gives us

\begin{equation}
    |F|^2 \pm 2ReF\Bar{g} + |g|^2 \leq 1
\end{equation}
and so
\begin{equation}
    |F|^2 + |g|^2 \leq 1
\end{equation}
Rewriting gives us the following inequality:
\begin{equation}
    |g|^2 \leq 1 - |F|^2 = (1 + |F|)(1 - |F|) \leq 2(1 - |F|)
\end{equation}
Recall that our aim here is to show that $g$ vanishes. Thus, define $h = \frac{1}{2}(g^2)$. Then, clearly, $h\in H^{\infty}$ and
\begin{equation}
    |F| + |h| \leq 1, \quad on \quad \Omega \label{hvan}
\end{equation}
The proof now consists of showing that \eqref{hvan} implies $h$ vanishes identically and this will show that $g$ does too.

First, $h(p) = 0$, which follows because otherwise, for some $\lambda$, s.t. $|\lambda| = 1$, we would have that

\begin{equation}
    (F + \lambda Fh)'(p) = F'(p) + \lambda F(p)h'(p) + \lambda F'(p)h(p) = 
\end{equation}
\begin{equation}
    = F'(p) + \lambda F'(p)h(p) > F'(p) = \gamma
\end{equation}
which is a contradiction, since
\begin{equation}
    |F + \lambda Fh| \leq |F| + |F||h| \leq |F| + |h| \leq 1
\end{equation}
Now assume that $h$ does not vanish identically. Let $r$ be the order of its zero at $p$, s.t., $r\geq 1$. Pick $\varepsilon$ to be a small complex number and let
\begin{equation}
    F(z) + \varepsilon h(z)(z-p)^{-r+1} = G(z)
\end{equation}
Now, the absolute value of the derivative of $G$ at the point $p$ is 
\begin{equation} \label{deriv}
    |G'(p)| = |\gamma + \varepsilon h^{(r)}(p)/r! | > \gamma
\end{equation}
for an appropriate choice of the argument of $\varepsilon$. We obtain equation \eqref{deriv} by using a Taylor series expansion. Outside of a neighbourhood of $p$, $|\varepsilon||z-p|^{1-r} < 1$ and by the maximum modulus principle, $|G|\leq 1$. This contradicts the previous statement that $|G'(p)|>\gamma$. Thus, $h\equiv 0$ and the proof is complete.
\end{proof}

\section{Properties of the Ahlfors function}
We are going to study some properties of the Ahlfors function which hold in the most general context.
\begin{prop}
If $H^{\infty}(\Omega)$ is non-trivial, then the Ahlfors function for the corresponding point $p$ is non-constant.
\end{prop}
\begin{proof}
Assume that $H^{\infty}(\Omega)$ is not trivial. It will be sufficient to show that there is at least one function $f\in H^{\infty}(\Omega)$ with a non-zero derivative at the point $p$, since this implies that

$$\Big|\frac{f'(p)}{\|f\|_{\infty}}\Big|>0$$ and $$\Big\|\frac{f}{\|f\|_{\infty}}\Big\|_{\infty}=1$$.

Hence, take $f\in H^{\infty}(\Omega)$, and define $g(z)=f(z)-f(p)$ for some $p\in\Omega$. Then, $g(z)\in H^{\infty}(\Omega)$ has a zero of order $n\geq 1$. If $n=1$, $g'(z)\neq 0$ and we are done. Now suppose that $n>1$. We can divide by $(z-p)^{1-n}$ so that $g'(z)\neq 0$. Hence, we have shown that there is a function in $H^{\infty}(\Omega)$ with a non-zero derivative.
\end{proof}

Note  that if $H^\infty(\Omega)$ is non-trivial then, from above, we have that the Ahlfors function $F$ corresponding to any $p\in \Omega$ is non-constant with modulus bounded by $1$, hence it satisfies $F(\Omega)\subset\mathbb{D}$.

\begin{cor} \label{almost surjective}
Assume that $H^{\infty}(\Omega)$ is not trivial. Fix $p\in\Omega$ and let $F$ be the corresponding Ahlfors function. If $G\subsetneq\mathbb{D}$ is any simply connected domain, then $F(\Omega)\setminus G \neq\emptyset$.
\end{cor}
\begin{proof}
First, let $f$ be analytic in $\Omega$ such that $f(p)=0$ and such that $f(\Omega)\subset G \subsetneq\mathbb{D}$, where $G$ is simply connected. Following the idea in the proof of the Riemann mapping theorem, we will find a function $H$, such that $H(p)=0$ but $|H'(p)|>|f'(p)|$. If we show this, then because the Ahlfors function maximizes the derivative, it cannot be that $F(\Omega)\subset G$ for any proper simply connected subset $G$ of the unit disk.

Let us construct $H$. Let $a\in \mathbb{D}\setminus G$. Then by assumption we have $f(z)\neq a$, $z\in\Omega$. Consider the automorphism $h_a$ of the disk that interchanges $0$ and $a$,
\begin{equation}
    h_a(z) = \frac{a - z}{1 - \overline{a}z}
\end{equation}
Then, we have that $h_a(G)$ is simply connected and does not contain 0. We can now define the analytic branch of the square root function on $h_a(G)$.

\begin{equation}
    g(w) = e^{\frac{1}{2}\log(w)}, \quad w\in h_{a}(G)
\end{equation}
Using this, we form the function 
\begin{equation}
  H = h_{g(a)}\circ g\circ h_a\circ f  
\end{equation}

where 

\begin{equation}
  h_{g(a)}(z) = \frac{e^{\frac{1}{2}\log(a)}-z}{1 - {e^{\frac{1}{2}\overline{\log(a)}}}z}  
\end{equation}
and where it is easy to see that $H(p)=0$.
We compute the derivative of $H$ at the point $p$. Let
$u(w)=w^2$ and 
$$\Phi = h_{a}^{-1}\circ u\circ h_{g(a)}^{-1}$$
Then,

$$f = \Phi\circ H$$
Taking the derivative gives us 
$$f' = (\Phi'\circ H)H'$$
so at $p$, we have
$$f'(p) = \Phi'(0)H'(p)$$
Note that $\Phi(0)=0$. Additionally, since $u$ is the square function, it is not injective and therefore $\Phi$ is also not injective. Hence, $\Phi(z) \neq cz$ for $|c|=1$, i.e., $\Phi$ is not a rotation. Thus, it follows by the Schwarz lemma (Lemma \ref{schwarz}) that $|\Phi'(0)|<1$ and 
$$|f'(p)|<|H'(p)|.$$
We have concluded that there exists a function $H$ with derivative at $p$ bigger than that of $f$, and this gives us the "almost" surjectivity of the Ahlfors function.
\end{proof}

Using this together with Theorem \ref{Fatou}, we arrive at the following remarkable property of the Ahlfors function.
\begin{thm}\label{composition} Assume that $H^\infty(\Omega)$ is not trivial,  fix $p\in \Omega$, and let $F$ denote the corresponding Ahlfors function.  Then for every $f\in H^\infty=H^\infty(\mathbb{D})$ we have 
$$\|f\circ F\|_\infty=\|f\|_\infty.$$
\end{thm}
\begin{proof} 
Let $f\in H^\infty$ and note first that $\|f\circ F\|_\infty\le\|f\|_\infty.$ This is because $F(\Omega)\subset\mathbb{D}$. For the reverse inequality, let 
$\varepsilon>0$. By Theorem \ref{Fatou}, we know that $\exists t\in [0,2\pi),~r_0\in (0,1)$, such that

$$|f(re^{it})|>\|f\|_\infty -\varepsilon,\quad r\in [r_0,1).$$

Now, we can apply Corollary \ref{almost surjective}. Choose $G=\mathbb{D}\setminus \{re^{it}: r\in [r_0,1)\}$. Then, $\exists q\in \Omega$ with $F(q)\in \{re^{it}: r\in [r_0,1)\}$. Hence,

$$\|f\circ F\|_\infty> |f(F(q))| > \|f\|_\infty-\varepsilon.$$
\end{proof}
This gives us, as a consequence, the non-separability of $H^\infty(\Omega)$ whenever this space is non-trivial.

\begin{cor}\label{nonsep}If $H^\infty(\Omega)$ is not trivial, then it is a non-separable Banach algebra.
\end{cor}
\begin{proof} 
In view of Theorem \ref{composition}, it suffices to show that $H^{\infty}(\mathbb{D})$ is not separable. Hence, let us consider the family of analytic functions 
\begin{equation}
    f_{s}(z) = \exp\Big(\frac{z + s}{z - s}\Big), \quad z\in \mathbb{D}, \quad s\in \partial\mathbb{D}
\end{equation}
Since $$\text{Re }\frac{z + s}{z - s}<0,\quad z\in \mathbb{D}, \quad s\in \partial\mathbb{D},$$
it follows that $f_s\in H^\infty$ with $\|f_s\|_\infty\le 1,~s\in \partial\mathbb{D}$. Moreover,  each $f_s$ extends analytically in $\mathbb{C}\setminus\{s\}$ and from
$$\text{Re }\frac{z + s}{z - s}=0,\quad z,s\in \partial\mathbb{D}, \quad z\neq s,$$
we obtain that $|f_s(z)|=1,~z,s\in \partial\mathbb{D}, ~z\neq s$. On the other hand, 
$$\lim_{r\rightarrow 1}f_s(rs) = \lim_{r\rightarrow 1}\exp\Big(\frac{rs + s}{rs - s}\Big) = \lim_{r\rightarrow 1}\exp\Big(\frac{(r+1)}{(r-1)}\Big) = 0$$
since $r$ approaches 1 from inside the disk. Thus, if $s_1,s_2\in \partial\mathbb{D}$, with $s_1\neq s_2$, we have
$$\|f_{s_1}-f_{s_2}\|_\infty\ge \lim_{r\rightarrow 1}|f_{s_1}(rs_2)-f_{s_2}(rs_2)|=1.$$ 
By Theorem \ref{composition} we obtain 
$$\|f_{s_1}\circ F-f_{s_2}\circ F\|_\infty\ge 1,\quad s_1,s_2\in \partial\mathbb{D}, \quad s_1\neq s_2.$$

This means that the open  balls of radius $1/3$ centered at $f_s\circ F,~s\in \partial\mathbb{D}$ in $H^\infty(\Omega)$  are disjoint. Since this family of nonvoid, open balls is uncountable, $H^\infty(\Omega)$ is not separable.
\end{proof}

\begin{thm}
Let $F$ be the Ahlfors function for $\Omega$ and $p\in\Omega$. Then for each $h\in H^{\infty}(\Omega)$ we have
\begin{equation} \label{bigproof1}
    \Vert{Fh}\Vert_{\infty} = \Vert{h}\Vert_{\infty}
\end{equation}
\end{thm}
\begin{proof}
To show equality \eqref{bigproof1}, we shall show first that $\Vert {Fh}\Vert_{\infty} \leq \Vert {h}\Vert_{\infty}$ and then that $\Vert {Fh}\Vert_{\infty} \geq \Vert {h}\Vert_{\infty}$. The first one follows immediately since, as we know, the Ahlfors function is bounded by 1. For the latter inequality, we need more arguments. Define $\Omega' = \{a\in \Omega:\exists h\in H^{\infty}(\Omega), s.t. |h(a)| > 1, \Vert{Fh}\Vert_{\infty} \leq 1 \}$. Our goal is to show that $\Omega'$ is empty. We do this by first showing that $\Omega'$ cannot be the entire space $\Omega$ and then that it is both open and closed by connectedness, and thus empty.
\newline
Suppose then that $p\in \Omega'$ and that $F_{1} = hF$ s.t. $\Vert{F_{1}}\Vert_{\infty} \leq 1$ and $|h(p)|>1$. Taking the derivative at $p$, we obtain

\begin{equation}
    |F_{1}'(p)| = |F(p)||h'(p)| + |F'(p)||h(p)| = |F'(p)||h(p)| > |F'(p)|
\end{equation}
which contradicts the definition of the Ahlfors function $F$. Thus $p\notin \Omega'$ that is,  $\Omega' \neq \Omega$.

\vspace*{5mm}
We show $\Omega'$ is open.  Let $a\in \Omega'$ and let $h\in H^{\infty}(\Omega)$ with $|h(a)|>1$, and $\|Fh\|_\infty\le 1$. Since $h$ is continuous at $a$, $\exists \epsilon > 0$ and a disk of radius $\varepsilon$, $\{|a-z|<\varepsilon\}$ centered at $a$ such that $\{|a-z|<\varepsilon\}\subset\Omega$ and $|h|>1$ on that disk, i.e., $\{|a-z|<\varepsilon\}\subset\Omega'$.

\vspace*{5mm}
To show $\Omega'$ is closed, we show that its complement is open. Let $q$ be a point in $\Omega \setminus\Omega'$. By definition,  we have for $h\in H^\infty(\Omega)$ that $\Vert{Fh}\Vert_{\infty}\leq 1$ implies $|h(q)|\leq 1$, which we rewrite as \begin{equation}\label{complement} |h(q)|\le  \|Fh\|_\infty,\quad h\in H^\infty(\Omega),\end{equation}
with equality for $h=1$.

Now  let $D\subset\Omega$ be an open disc centered at $q$  with $\overline{D}\subset\Omega$ and consider the Banach algebra $H^\infty(\Omega\setminus\overline{D})$. Obviously, this algebra contains $H^\infty(\Omega)$ and by the maximum principle (Theorem \ref{maxmod}) we have $$\sup_{z\in \Omega\setminus\overline{D}} |g(z)|=\|g\|_\infty,\quad g\in H^\infty(\Omega).$$ In functional-analytic terms, this means that the linear functional $l$ defined on the subspace
$FH^\infty(\Omega)\subset  H^\infty(\Omega\setminus\overline{D})$ by
$$l(Fh)=h(q), \quad h\in H^\infty(\Omega),$$ satisfies
$$|l(g)|\le \|g\|_\infty,\quad g\in FH^\infty(\Omega),$$
where the   supremum norm is computed in $\Omega\setminus\overline{D}$. Moreover, equality holds when $g=F$.  Then by the Hahn-Banach theorem  (Theorem \ref{Hahn-Banach}) $l$ has an extension of norm $1$ to the whole algebra $H^\infty(\Omega\setminus\overline{D})$ which will also be denoted by $l$. From \eqref{representation} we have that $l$ can be written as 
$$l(f)= \int_{\textgoth{M}}\hat{f}d\lambda, \quad f\in H^\infty(\Omega\setminus\overline{D}),$$
where $\textgoth{M}$ denotes the maximal ideal space of  $ H^\infty(\Omega\setminus\overline{D})$, $\lambda$ is a finite Borel measure on this compact Hausdorff space of total variation $1$, and $\hat{f}$ denotes, as usual, the Gelfand transform of 
$f\in H^\infty(\Omega\setminus\overline{D})$.
Since $|\hat{F}|\le\|F\|_\infty=1$, we infer that 
\begin{equation}
    1= l(F) = \int_{\textgoth{M}}\hat{F}d\lambda \leq \int_{\textgoth{M}}|\hat{F}|d|\lambda| \leq \Vert{\lambda}\Vert = 1 \label{integ}
\end{equation}

Now, we can consider the measure $d\mu = \hat{F}d\lambda$. Let us note some properties of this measure.
\begin{enumerate}
    \item The measure $d\mu$ is non-negative and has total mass $1$. This follows from equation \eqref{integ}, which shows that $1=\mu(\textgoth{M})$ equals its total variation.
    \item $\mu$ is supported on the set where $|\hat{F}| = 1$, since otherwise the integral in \eqref{integ} is strictly less than $1$.
\end{enumerate}

Now take a point $y\in E$, where $E\subset D$ is an open disk centered at $q$ or radius half of that of $D$, and define the auxiliary function
\begin{equation}
    s(z,y) = \frac{(z-q)}{(z-y)}, z\in \Omega\setminus D
\end{equation}
Note that $s(\cdot,y)\in   H^\infty(\Omega\setminus\overline{D})$ and define $u(y)$,
\begin{equation}
    u(y) = \int_{\textgoth{M}}\hat{s}d\mu
\end{equation}
It is easy to verify that if $y\to y_0$ in $E$, 
$s(\cdot,y)$ converges to $s(\cdot,y_0)$ in $H^\infty(\Omega\setminus  \overline{D})$, which implies that $u$ is continuous on $E$. Now let $h\in H^{\infty}(\Omega)$ and write
\begin{equation}
    g(z,y) = \frac{h(z)-h(y)}{z-y}(z-q) \label{g(z)}
\end{equation}
Then $g(\cdot,y)\in H^\infty(\Omega)$, and by the definition of $\mu$, we have that
\begin{equation} \label{representation1}
    \int_{\textgoth{M}} \hat{g}(\cdot,y)d\mu = \int_{\textgoth{M}} \hat{F}\hat{g}(\cdot,y)d\lambda = g(q,y) = 0
\end{equation}
since $g(q)=0$ by \eqref{g(z)}.

\vspace*{1cm}
Observe that $g(z,y)=(h(z)-h(y))s(yz,y)$. Hence, since the Gelfand transform is linear and  multiplicative,   we have $\hat{g}(\cdot,y)=(\hat{h}-h(y))\hat{s}(\cdot,y)$, and  \eqref{representation1} gives 
$$ \int_{\textgoth{M}}(\hat{h}-h(y))\hat{s}(\cdot,y)d\mu=0,$$
or equivalently,
\begin{equation}
    \int_{\textgoth{M}} \hat{h}\hat{s}(\cdot,y)d\mu = h(y)u(y).
\end{equation}
Also recall that $u$ is continuous on $E$ with $u(q)=1$. Therefore, there exists a neighbohood $V$ of $q$ such that $|u(y)|>1/2,~y\in V$, which leads to the estimate 
\begin{equation}
    |h(y)| \leq 2\int_{\textgoth{M}}|\hat{h}\hat{s}(\cdot,y)|d\mu \leq C\int_{\textgoth{M}}|\hat{h}|d\mu, \quad y\in V,
\end{equation}
where $C=2\|s\|_\infty$.
This implies for $y\in V$,
$$ |h(y)| \le C\sup_{\textgoth{M}}\{|\hat{h}(m)|: m \in A \}
$$
where $A$ denotes the  support of $\mu$. We now observe that the estimate actually holds with $C=1$. Indeed, for all positive integers $n$, replacing $h$ by $h^n$, we obtain 
\begin{equation}
    |h^{n}(y)| \leq C\sup\{|\hat{h}^{n}(m)|: m \in A \} = C(\sup_{\textgoth{M}}\{|\hat{h}(m)|: m \in A \})^{n}
\end{equation} so that 
\begin{equation}
    |h(y)| \leq C^{1/n} \sup\{|\hat{h}(m)|: m \in A \} \label{equality}
\end{equation}
and letting $n\to\infty$ we get 
\begin{equation}
    |h(y)| \leq \sup\{|\hat{h}(m)|: m \in A \},\quad y\in V. 
\end{equation}
Finally,  recall that $A\subset \{m\in \textgoth{M}:~|\hat{F}(m)|=1\}$, so that 
$$ |h(y)| \le \sup\{|\hat{h}(m)|: m \in A \}= \sup\{|\widehat{Fh}(m)|: m \in A \}\le \|Fh\|_\infty,$$
for all $y\in V$, that is $V$ is contained in $\Omega\setminus\Omega'$. This  shows that $\Omega'$ is closed and the proof is complete.
\end{proof}

Recall that by the Riemann mapping thoerem (Theorem \ref{riemannmap}), if $\Omega\neq \mathbb{C}$ is simply connected  and $p\in \Omega$, there exists a unique conformal map $\psi$ from $\Omega$ onto $\mathbb{D}$ with $\psi(p)=0,~\psi'(p)>0$. Moreover, $\psi$ appears as the solution of an extremal problem which is very similar to the one defining the Ahlfors function. We shall call $G$ the {\em Riemann map corresponding to $p$}.
\begin{thm}\label{Riemann}
For a simply connected domain $\Omega\neq \mathbb{C}$, the Ahlfors function  corresponding to $p\in \Omega$ equals  the Riemann map corresponding to $p$.
\end{thm}
\begin{proof}
 Let  $p\in\Omega$ and let $F$ be the corresponding Ahlfors function. The idea is to compose the Ahlfors function with a Riemann map  from the disc onto $\Omega$. 
 Let $\psi$ be the Riemann map corresponding to $p$, and let $\phi=\psi^{-1}, \psi:\Omega\to\mathbb{D}$.
Then $\phi(0)=p$ and since $\psi(\phi(z))=z,~  z\in \mathbb{D}$, $$\phi'(0)=\frac1{\psi'(p)}.$$
The composition $F\circ\phi$ maps $\mathbb{D}$ into itself and  satisfies $F\circ\phi(0)=0$. Then by  the Schwarz lemma (Theorem \ref{schwarz})  we have that
$$|(F\circ\phi)'(0)|=\left|\frac{
F'(p)}{\psi'(p)}\right|\le 1.$$
On the other  hand,  $\psi$ is in the unit ball of $H^\infty(\Omega)$, hence $0<\psi'(p)\le F'(p)$. Thus $\psi'(p)=F'(p)$ and by the uniqueness of the Ahlfors function it follows that $F=\psi$. 
\end{proof}

\begin{cor}\label{disk}
Let $\Omega=\mathbb{D}$ and $p\in \mathbb{D}$. Then  the Ahlfors function corresponding to $p$ is  $F(z)=\frac{z-p}{1-\overline{p}z}$.
\end{cor}
\begin{proof} This follows directly from Theorem \ref{Riemann} since the automorphism in the
statement is the Riemann map from $\mathbb{D}$ onto itself, corresponding to the point $p$.
 \end{proof}

\begin{ex} \label{exampleahl}
Let $\Omega=\mathbb{D}^{*}$, i.e., the exterior of the closed unit disk in the Riemann sphere. Then, the Ahlfors function on this domain for the point $p=2$ will be given by $F(z)=\frac{z-2}{2z-1}$. It is obtained by composing the automorphism in Corollary \ref{disk} for  $p=\frac{1}{2}$ with the function $\frac{1}{z}$. 
\end{ex}

\begin{ex}
Let $E$ be a compact subset of the real line of positive Lebesgue measure and let $\Omega = \textbf{S}^2\setminus E$. Then, we can write explicitly the Ahlfors function for $\Omega$ corresponding to $p=\infty$. For this purpose, define
$$h(z) = \frac{1}{2}\int_{E}\frac{dt}{z-t}, z\notin E$$
The function $h(z)$ maps $\textbf{S}^2\setminus E$ into the horizontal strip $|\operatorname{Im}(z)|<\frac{\pi}{2}$. One can see this by observing that
$$\operatorname{Im}\Big(\frac{1}{z-t}\Big)=\frac{-y}{(x-t)^2+y^2}$$ 
and so we obtain the following inequality
$$|\operatorname{Im(h(z))}|=\frac{1}{2}\Big|\int_{E}\frac{y}{(x-t)^2+y^2}dt\Big| < \frac{1}{2}\int_{\mathbb{R}}\frac{|y|}{(x-t)^2+y^2}dt = \frac{\pi}{2}$$
Hence, the function

$$F(z) = \frac{e^{h(z)}-1}{e^{h(z)}+1}$$
maps $\Omega$ into the open unit disk and since $\operatorname{Re}(e^{h(z)})>0$, we have that $|F(z)|<1$. Additionally, $F(\infty)=0$. We can take the derivative to obtain
$$F'(z) = \frac{-e^{h(z)}\int_{E}\frac{1}{(z-t)^2}dt}{(e^{h(z)}+1)^2}$$
and at $z=\infty$, we get
$$F'(\infty)=\frac{1}{4}\int_{E}dt = \frac{1}{4}\lambda(E),$$
where $\lambda(E)$ is the Lebesgue measure of the set $E$. Thus, we have that $\gamma(\Omega)\geq F'(\infty)$. In \cite{pommerenke}, C. Pommerenke showed that $\gamma(\Omega)=\frac{1}{4}\int_{E}dt$, which means that $F$ is the Ahlfors function for $\Omega$ and $\infty$.
\end{ex}
We close with a general theorem of Ahlfors which refers to domains with real-analytic boundary and can be seen as a generalization of Theorem \ref{Riemann}.  The proof of this result is fairly involved and can be found in  \cite{fisher},  (Chapter 5, Theorem 1.6)
\begin{thm}
Let $\Omega$ be bounded by $m+1$ disjoint analytic simple closed curves denoted by $\Gamma_0,...,\Gamma_m$. Let $F$ be the Ahlfors function for $\Omega$ and $p\in\Omega$. Then 
\begin{itemize}
       \item F extends analytically across each $\Gamma_j$ and maps  $\Gamma_j$ homeomorphically onto the unit circle.
    \item $F'$ is not zero on any $\Gamma_j$.
\item $F$ takes any  value in $\mathbb{D}$ precisely $m+1$ times.
\end{itemize}
\end{thm}

\bibliographystyle{unsrt}
\bibliography{sample}


\end{document}